\documentclass[12pt,reqno]{amsart}

\usepackage{ytableau}
\usepackage{latexsym,amsmath,amssymb}
\usepackage[normalem]{ulem}
\usepackage{etoolbox}

\usepackage{hyperref}
\hypersetup{
	colorlinks=true, 
	linktoc=all, 
	linkcolor=blue} 

\DeclareMathOperator{\vl}{\raisebox{-3pt}{$\overset{L}{\sim}$\,}}
\DeclareMathOperator{\f}{\raisebox{-3pt}{$\overset{F}{\sim}$\,}}
\DeclareMathOperator{\g}{\raisebox{-3pt}{$\overset{G}{\sim}$\,}}
\DeclareMathOperator{\s}{\raisebox{-3pt}{$\overset{S}{\sim}$\,}}

\renewcommand{\o}[1]{\overline{#1}}

\allowdisplaybreaks[4]

\numberwithin{equation}{section}

\newtheorem{theorem}{Theorem}
\newtheorem*{theorem*}{Theorem}

\theoremstyle{definition}
\newtheorem{definition}{Definition}

\newtheorem*{example*}{Example}

\newtheorem*{conjecture*}{Conjecture}

\newtheorem*{remark*}{Remark}
\newtheorem*{remarks*}{Remarks}

\makeatletter
\patchcmd{\section}{\scshape}{\bfseries\boldmath}{}{}
\patchcmd{\subsection}{\bfseries}{\bfseries\boldmath}{}{}
\renewcommand{\@secnumfont}{\bfseries}
\makeatother

\makeatletter
\newcommand{\oset}[3][0ex]{%
  \mathrel{\mathop{#3}\limits^{
    \vbox to#1{\kern-2\ex@
    \hbox{$\scriptstyle#2$}\vss}}}}
\makeatother

\begin{document}
	
\title[Overpartitions Wherein Only Even Parts Are Overlined]{Enumeration Modulo 4 of Overpartitions Wherein Only Even Parts May Be Overlined}

\author[A.\ Carlson]{Aidan Carlson}
\address[A.\ Carlson]{Department of Mathematics and Statistics, University of Minnesota Duluth, Duluth, MN 55812, USA}
\email{carl6746@umn.edu}
	
\author[B.\ Hopkins]{Brian Hopkins}
\address[B.\ Hopkins]{Department of Mathematics and Statistics, Saint Peter's University, Jersey City, NJ, USA}
\email{bhopkins@saintpeters.edu}

\author[J.\ A.\ Sellers]{James A.\ Sellers}
\address[J.\ A.\ Sellers]{Department of Mathematics and Statistics, University of Minnesota Duluth, Duluth, MN 55812, USA}
\email{jsellers@d.umn.edu}

	

\begin{abstract}
In 2014, as part of a larger study of overpartitions with restrictions of the overlined parts based on residue classes, Munagi and Sellers defined $d_2(n)$ as the number of overpartitions of weight $n$ wherein only even parts can be overlined.  As part of that work, they used a generating function approach to prove a parity characterization for $d_2(n)$.  In this note, we give a combinatorial proof of their result and extend it to a modulus 4 characterization; we provide both generating function and combinatorial proofs of this stronger result.  The combinatorial arguments incorporate classical involutions of Franklin, Glaisher, and Sylvester, along with a recent involution of van Leeuwen and methods new with this work.
\end{abstract}
\maketitle

\section{Introduction}  
A partition $\lambda$ of a positive integer $n$ is a finite sequence of positive integers $\lambda_1 \geq \lambda_2 \geq \dots \geq \lambda_t$ such that $\sum \lambda_i = n$, sometimes called the weight $|\lambda|$ of $\lambda$.  We refer to the integers $\lambda_i$ as the parts and the number of parts $t$ as the length of the partition.  For a given $n$, write $P(n)$ for the set of partitions of $n$ and let $p(n) = |P(n)|$, a notational convention we use throughout the paper.  For example,   
$$P(4) = \{(4), (3,1), (2,2), (2,1,1), (1,1,1,1)\}$$
and $p(4) = 5$.  

There are several subsets of partitions we will consider.  Let $Q(n)$ denote the partitions of $n$ into distinct parts; e.g., $Q(4) = \{(4), (3,1)\}$ and $q(4) = 2$.  Let $P^e(n)$ denote the partitions of $n$ with even length, likewise $P^o(n)$ for odd length.  In contrast, let $P_e(n)$ denote the partitions of $n$ consisting only of even parts, likewise $P_o(n)$ for odd parts.  For example, $P^e(4) = \{(3,1),(2,2),(1,1,1,1)\}$ and $p_o(4) = 2$ from $(3,1)$ and $(1,1,1,1)$ (which we sometimes write as $(1^4)$).  The same notations apply to any set of partitions, e.g., $q^o(4) = 1$ from the partition $(4)$. 

We will assume the reader is familiar with standard concepts related to integer partitions, such as Ferrers diagrams, conjugation (denoted $\lambda'$), hook lengths, and the notation $(q;q)_m$, all covered, for example, in the book of Andrews and Eriksson \cite{AE}.

An overpartition of a positive integer $n$ is a partition of $n$ wherein the first occurrence of a part may be overlined.  
For example, there are 14 overpartitions of 4:  
\begin{gather*}
(4), (\o{4}), (3,1),  (\o{3},1), (3,\o{1}), (\o{3},\o{1}), (2,2),  (\o{2},2),\\
(2,1,1), (\o{2},1,1), (2,\o{1},1), (\o{2},\o{1},1), (1,1,1,1), (\o{1},1,1,1).
\end{gather*}
These were named by Corteel and Lovejoy in 2004 \cite{CL} although, as they discuss, equivalent sets of partitions had previously been considered in various contexts.

Note that an overpartition of $n$ can be viewed as a bipartition $(\mu, \nu)$ with $|\mu| + |\nu| = n$ where $\mu \in P(m)$ for some $m \le n$ corresponds to the nonoverlined parts of the overpartition and $\nu \in Q(n-m)$ corresponds to the overlined parts of the overpartition. 

In 2014, as part of a larger study of overpartitions with restrictions on the overlined parts based on residue classes, Munagi and Sellers \cite{MS} refined the idea of an overpartition by considering $D_2(n)$, the overpartitions of $n$ wherein only even parts can be overlined.  From the previous example, we have $d_2(4) = 8$ and 
$$D_2(4) =  \{(4), (\o{4}), (3,1), (2,2), (\o{2},2), (2,1,1), (\o{2},1,1), (1,1,1,1)\}.$$
In terms of the bipartitions $(\mu, \nu)$, partitions in $D_2(n)$ have $\mu \in P(m)$ for some $m \le n$ and $\nu \in Q_e(n-m)$.

The sequence of $d_2(n)$ values is included in the On-Line Encyclopedia of Integer Sequences \cite[A279328]{OEIS}.  It is straightforward to see that a generating function for $d_2(n)$ is
\begin{equation}
\label{d2gf}
\sum_{n=0}^\infty d_2(n)q^n = \left( \prod_{i\geq 1} 1+q^{2i} \right)  \left( \prod_{i\geq 1} \frac{1}{1-q^i} \right) = \prod_{i\geq 1} \frac{1+q^{2i}}{1-q^i}.
\end{equation}
 
We note, in passing, that this generating function can be interpreted in other ways.  For example, it is easy to see that 
\[ \prod_{i\geq 1} \frac{1+q^{2i}}{1-q^i} = \prod_{i\geq 1} \frac{1-q^{4i}}{(1-q^i)(1-q^{2i})} = \prod_{i\geq 1} \frac{1}{(1-q^i)(1-q^{4i-2})}. \]
Thus, $d_2(n)$ is also the number of partitions of $n$ where parts 2 modulo 4 are allowed to be colored in two different ways.  

We also note that, in the recent work of Kur\c{s}ung\"{o}z and Seyrek \cite{KS}, the authors considered what are called cylindrical partitions with a given profile.  They showed that the generating function for the number of cylindrical partitions with profile $(2,0)$ matches \eqref{d2gf}.

In their 2014 work, Munagi and Sellers \cite[Theorem 4.11]{MS} proved the following parity characterization satisfied by $d_2(n)$ via an elementary generating function manipulation.

\begin{theorem}[Munagi--Sellers]
\label{d2mod2}
For all $n\geq 0$, 
$$d_2(n) \equiv \begin{cases} 
     1 \bmod 2 & \text{if $n = m(3m+1)/2$ for some integer $m$},\\ 
     0 \bmod 2 & \text{otherwise}.
     \end{cases}$$
\end{theorem}

Because the generalized pentagonal numbers will arise frequently, let $\omega_m = m(3m+1)/2$.

Our primary goal in this paper is to extend the above parity result to the following characterization modulo 4.
\begin{theorem}
\label{d2mod4}
For all $n\geq 0$, 
     $$d_2(n) \equiv \begin{cases} 
     1 \bmod 4 & \text{if $n = \omega_{4k}$ or $n = \omega_{4k+3}$},\\ 
     3 \bmod 4 & \text{if $n = \omega_{4k+1}$ or $n = \omega_{4k+2}$},\\ 
     0 \bmod 4 & \text{otherwise}.
     \end{cases}$$
\end{theorem}

In the work below, after giving a combinatorial proof of Theorem \ref{d2mod2}, we provide both generating function and combinatorial proofs of Theorem \ref{d2mod4}.  Together, the combinatorial proofs incorporate classical partition maps of Franklin, Glaisher, and Sylvester, along with a newer one by Marc van Leeuwen, which we discuss in the next section.

\section{Partition Involutions}
Here we review several classical partition involutions and introduce a relatively new one into the literature.  

\subsection{Involutions from the 1880s}
We will use three involutions from the early 1880s, when Sylvester began the combinatorial study of partitions.  The first two relate to results of Euler.

Euler used generating functions to prove the first result establishing the equal count of two types of restricted partitions, namely $p_o(n) = q(n)$,
the ``odd-distinct'' partition identity.  In 1883, Glaisher generalized this identity, providing both generating function and combinatorial proofs of his result \cite{G}.  We will only need Glaisher's combinatorial method for Euler's result: In a partition with distinct parts, even parts are split into halves until all parts are odd.  In a partition with odd parts, any two repeated parts are merged into a single part until all parts are distinct.  Table \ref{gtab} gives Glaisher's example for the involution; see \cite[p.\ 6]{A} and \cite[pp.\ 8--9]{AE} for further details.  Note that the fixed points of Glaisher's map are $Q_o(n)$, the partitions of $n$ consisting of distinct odd parts.

We write $(3,1^6) \g (4,3,2)$ to indicate partitions connected by Glaisher's involution.  Equivalently, we will sometimes denote that same relation by $G((3,1^6)) = (4,3,2)$ and $G((4,3,2)) = (3,1^6)$.  

\begin{table}[th]
\renewcommand{\arraystretch}{1.25}
\begin{tabular}{r|l}
$P_o(9)$ & $Q(9)$ \\ \hline
$(9)$ &$(9)$ \\
$(7,1,1)$ & $(7,2)$ \\
$(5,3,1)$ & $(5,3,1)$ \\
$(5,1^4)$ &  $(5,4)$ \\
$(3,3,3)$ & $(6,3)$ \\
$(3,3,1,1,1)$ & $(6,2,1)$ \\
$(3,1^6)$ & $(4,3,2)$ \\
$(1^9)$ & $(8,1)$
\end{tabular}
\caption{Glaisher's involution for $n = 9$.} \label{gtab}
\end{table}

Another celebrated result of Euler is the pentagonal number theorem,
\begin{equation}
\prod_{i\geq 1} (1-q^i) = \sum_{j=-\infty}^\infty (-1)^jq^{j(3j+1)/2} \label{penteq}
\end{equation}
(see Hirschhorn \cite[\S 1.6]{H} for a proof), which leads to a recursive formula to compute $p(n)$.  Legendre gave the following equivalent formulation of the pentagonal number theorem \cite[\S458]{L}:
\begin{equation}
q^e(n) - q^o(n) = \begin{cases} (-1)^m & \text{if $n = \omega_m$,} \\ 0 & \text{otherwise.} \end{cases} \label{legeq}
\end{equation}
For example, from the right-hand side of Table \ref{gtab}, one can see that $q^e(9) = q^o(9) = 4$.
In 1881, Franklin gave an ingenious combinatorial proof of Legendre's version of the result, a near-bijection that leaves one partition unpaired when $n$ is a generalized pentagonal number \cite{F}.  The map involves the last part of a partition and the connected diagonal of rightmost boxes starting in the first row of the Ferrers diagram; Figure \ref{fex} shows $(6,5,3) \f (5,4,3,2)$ and \cite[pp. 9--13]{A} and \cite[pp. 24--27]{AE} provide full details.

\begin{figure}[th]
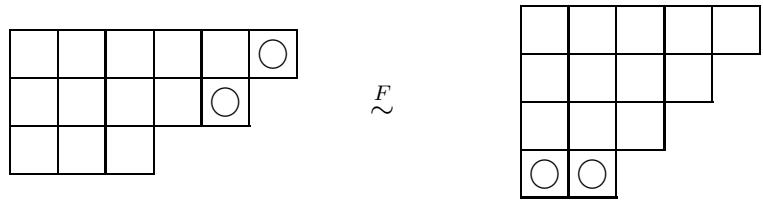

\centering
\ytableausetup{centertableaux}
\begin{ytableau}
\none & & & & & & \bigcirc \\
\none & & & & & \bigcirc  \\
\none & & & 
\end{ytableau}
\qquad $\f$ \qquad
\begin{ytableau}
\none & & & & &  \\
\none & & & & \\
\none & & &  \\
\none & \bigcirc & \bigcirc 
\end{ytableau}
\caption{Franklin's map pairs $(6,5,3)$ and $(5,4,3,2)$.} \label{fex}
\end{figure}

Table \ref{ftab} gives Franklin's correspondence for $Q(12)$. Note that the unpaired partition $(5,4,3)$ has an odd number of parts so that $q^e(12) - q^o(12) = -1$, consistent with $12 = \omega_{-3}$.  In the case of the next smaller generalized pentagonal number, $w_2 = 7$, the unpaired partition $(4,3)$ has even length and one can verify from the five partitions of $Q(7)$ that $q^e(7) - q^o(7) = 1$.

\begin{table}[th]
\renewcommand{\arraystretch}{1.25}
\begin{tabular}{r|l}
$Q^e(12)$ & $Q^o(12)$ \\ \hline
$(11,1)$ &$(12)$ \\
$(10,2)$ & $(9,2,1)$ \\
$(9,3)$ & $(8,3,1)$ \\
$(8,4)$ &  $(7,4,1)$ \\
$(7,5)$ & $(6,5,1)$ \\
$(6,3,2,1)$ & $(7,3,2)$ \\
$(5,4,2,1)$ & $(6,4,2)$ \\
 & $(5,4,3)$
\end{tabular}
\caption{Franklin's near-pairing of $Q(12)$.}  \label{ftab}
\end{table}

The last classical involution we will use is one from Sylvester.  In his 1882 magnum opus \cite{S}, he established the identity
\[p( n \mid \text{self-conjugate}) = q_o(n),\]
i.e., the number of $\lambda \in P(n)$ with $\lambda = \lambda'$ equals the number of partitions of $n$ into distinct odd parts.  The correspondence simply makes the diagonal hooks of a self-conjugate partition into parts.  Figure \ref{sex} shows $(5,3,3,1,1) \s (9, 3, 1)$; see \cite[p. 14]{A} and \cite[p. 18]{AE} for more details.  Table \ref{stab} gives Sylvester's correspondence when $n = 20$.

\begin{figure}[th]
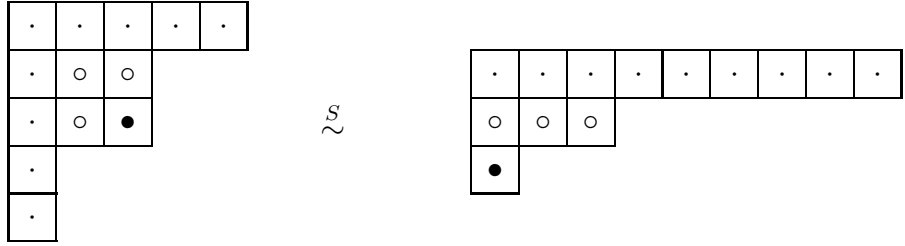

\centering
\ytableausetup{centertableaux}
\begin{ytableau}
\none & \cdot & \cdot & \cdot & \cdot & \cdot  \\
\none & \cdot & \circ &  \circ \\
\none & \cdot & \circ &  \bullet \\
\none & \cdot \\
\none & \cdot
\end{ytableau}
\qquad $\s$ \qquad
\begin{ytableau}
\none & \cdot & \cdot & \cdot & \cdot & \cdot & \cdot & \cdot & \cdot & \cdot \\
\none & \circ & \circ & \circ \\
\none & \bullet
\end{ytableau}
\caption{Sylvester's map pairs $(5,3,3,1,1)$ and $(9,3,1)$.} \label{sex}
\end{figure}

\begin{table}[th]
\renewcommand{\arraystretch}{1.25}
\begin{tabular}{r|l}
 $P(20 \mid \text{self-conjugate})$ & $Q_o(20)$ \\ \hline
$(10,2,1^8)$ &	$(19,1)$ \\
$(9,3,2,1^6)$ &  $(17,3)$ \\
$(8,4,2,2,1^4)$ & $(15,5)$ \\
$(7,5,2^3,1,1)$ & $ (13,7)$ \\
$(6,6,2^4)$ & $(11,9)$ \\
$(6,4^3,1,1)$ & $(11,5,3,1)$ \\
$(5,5,4,4,2)$ & $(9,7,3,1)$
\end{tabular}
\caption{Sylvester's involution for $n = 20$.} \label{stab}
\end{table}

\subsection{Van Leeuwen's involution}
A much more recent involution was first devised (we believe) in 2011 by Marc van Leeuwen as part of an answer to a question on the site Mathematics Stack Exchange \cite{vL}.

Let $R(n) = P(n) \setminus Q_o(n)$, the partitions of $n$ that do not consist of distinct odd parts, with $r^e(n)$ and $r^o(n)$ denoting the number of partitions in $R(n)$ with even (respectively, odd) length.  Van Leeuwen's involution, a more subtle split and merge procedure than Glaisher's bijection, provides a combinatorial proof of the identity $r^e(n) = r^o(n)$ for each $n \ge 2$.

Given $\lambda \in R(n)$ for $n \ge 2$:
\begin{itemize}
\item Let $m$ be the least positive odd integer such that
the set of parts of the form $2^km$ for $k \ge 0$
is nonempty and not a single part $m$.
\end{itemize}
(Since $\lambda$ is not the empty partition, there are such parts for at least one odd number $m$.  Since $\lambda \notin Q_o(n)$, it cannot be the case that the list of parts is a single $m$ for every odd number $m$.)  
\begin{itemize}
\item Choose the maximal $k$ such that some part $\lambda_i = 2^km$.  
\item If the part $2^km$ is unique in $\lambda$, then we know $k \ne 0$ by the choice of $m$; split $\lambda_i$ into two new parts each $2^{k-1} m$.  
\item If $2^km$ appears two or more times in $\lambda$, then merge two of those parts to make a new part $2^{k+1}m$.
\end{itemize}
In either case, the resulting partition is in $R(n)$ and the length has changed by one.  For example, $(6,6,5,5,1)$ has $m = 3$ and the two parts 6 are merged, therefore $(6,6,5,5,1) \vl (12,5,5,1)$.  See Table \ref{ltab} for van Leeuwen's involution when $n = 8$.

\begin{table}[th]
\renewcommand{\arraystretch}{1.25}
\begin{tabular}{r|l}
$R^e(8)$ & $R^o(8)$ \\ \hline
$(6,2)$ &	$(6,1,1)$ \\
$(5,1^3)$ & $(5,2,1)$\\
$(4,4)$ &  $(8)$ \\
$(4,2,1,1)$ & $(2^3,1,1)$ \\
$(3,3,1,1)$ & $ (3,3,2)$ \\
$(3,2,2,1)$ & $(4,3,1)$ \\
$(3,1^5)$ & $(3,2,1^3)$ \\
$(2^4)$ & $(4,2,2)$ \\
$(2,2,1^4)$ & $(4,1^4)$ \\
$(1^8)$ & $(2,1^6)$
\end{tabular}
\caption{Van Leeuwen's pairing of $R(8)$.} \label{ltab}
\end{table}

\section{Combinatorial Proof of Theorem \ref{d2mod2}}
Using Sylvester's involution described above, we give a combinatorial proof of Theorem \ref{d2mod2}.  
This complements the existing generating function proof of \cite[Theorem 4.11]{MS} and introduces a variant of conjugation on overpartitions (Corteel and Lovejoy \cite[p. 1628]{CL} define overpartition conjugation in a different way).

\begin{definition}
The half-conjugate of an overpartition fixes any overlined parts and conjugates the subpartition consisting of any nonoverlined parts.
\end{definition}
For example, the half-conjugate of $(3,\o{2},1)$ is $(\o{2},2,1,1)$ since the overlined part $\o{2}$ is fixed and $(3,1)' = (2,1,1)$.  Since conjugation is an involution, so is half-conjugation.  Since the overlined parts do not change, $D_2(n)$ is closed under half-conjugation.  Therefore, half-conjugation partitions $D_2(n)$ into half-conjugate pairs and self-half-conjugate singletons.

Let $H(n)$ be the set of self-half-conjugate partitions in $D_2(n)$.  The left-hand side of Table \ref{hctab} gives $H(8)$.  Computing $H(n)$ for small $n$ suggests that $h(n) = q(n)$.  The next theorem gives a bijective proof of this relation.

\begin{theorem} \label{hq}
For all $n\geq 1$, the number of self-half conjugate partitions in $D_2(n)$ equals the number of partitions of $n$ into distinct parts, i.e., $h(n) = q(n)$.    
\end{theorem}
    
\begin{proof}
We establish a bijection $H(n) \cong Q(n)$.

Write $\lambda \in H(n)$ as a bipartition $(\mu, \nu)$ where $\mu$ consists of the nonoverlined parts of $\lambda$ and $\nu$ corresponds to the overlined parts of $\lambda$.  Since $\lambda$ is self-half-conjugate, we know $\mu = \mu'$.

We claim that the partition with parts $S(\mu) \cup \nu$ is in $Q(n)$.  Sylvester's involution takes the self-conjugate $\mu$ to a partition of distinct odd parts, i.e., $S(\mu) \in Q_o(m)$ for some $m \le n$.  By the definition of $D_2(n)$, we know that $\nu$ consists of distinct even parts, so $\nu \in Q_e(n-m)$.  Therefore $S(\mu) \cap \nu = \varnothing$ and $S(\mu) \cup \nu$ gives a distinct part partition of $n$.

For the reverse map, given $\pi \in Q(n)$, partition it as $(\rho, \sigma)$ where $\rho$ consists of the odd parts of $\pi$ and $\sigma$ consists of the even parts of $\pi$.  We claim that the bipartition $(S(\rho), \sigma)$ corresponds to a partition in $H(n)$.  Since $\rho$ consists of distinct odd parts, Sylvester's involution can be applied and $S(\rho)$ is self-conjugate.  Associate $\sigma$, consisting of distinct even parts, with the overlined parts which are fixed by half-conjugation.

The two maps are clearly inverses, establishing the bijection, and the enumeration result follows.
\end{proof}

See Table \ref{hctab} for the $n = 8$ case of the bijection.

\begin{table}[th]
\renewcommand{\arraystretch}{1.25}
\begin{tabular}{r|l}
$H(8)$ & $Q(8)$ \\ \hline
$(\o{8})$ & $(8)$ \\
$(\o{6}, \o{2})$ & $(6, 2)$ \\
$(\o{4},2,2)$ & $(4,3,1)$ \\
$(3,\o{2},2,1)$ & $(5,2,1)$ \\
$(4,2,1,1)$ & $(7,1)$ \\
$(3,3,2)$ & $(5,3)$
\end{tabular}
\caption{The bijection of Theorem \ref{hq} for $n = 8$.} \label{hctab}
\end{table}

Now we can immediately provide a combinatorial proof of Theorem \ref{d2mod2} as a corollary of Theorem \ref{hq}.  

\begin{proof}[Proof of Theorem \ref{d2mod2}] 
For all $n$, note that $d_2(n) \equiv h(n) \bmod{2}$ since removing the half-conjugate pairs from $D_2(n)$ does not change the parity of $d_2(n)$.  By Theorem \ref{hq} we have $d_2(n) \equiv q(n) \bmod{2}$ and the result follows from \eqref{legeq}.  \end{proof}

\section{Proofs of Theorem \ref{d2mod4}}
We now prove our primary result, Theorem \ref{d2mod4}, which extends the original modulo 2 result of Munagi and Sellers to modulus 4.

The first proof uses the generating function for $d_2(n)$ \eqref{d2gf}, Euler's pentagonal number theorem \eqref{penteq}, and the following result, which follows essentially from the binomial theorem: For a prime $p$ and positive integers $k$ and $\ell$,
\begin{equation}
(1-q^k)^{p^\ell} \equiv (1-q^{pk})^{p^{\ell-1}} \bmod{p^\ell}. \label{bin}
\end{equation}

\begin{proof}[First proof of Theorem \ref{d2mod4}]
Consider \eqref{d2gf} with $q$ replaced by $-q$ throughout:
\begin{align*}
\sum_{n=0}^\infty d_2(n)(-q)^n  &= \prod_{i\geq 1} \frac{1+q^{2i}}{1-(-q)^i} \\
&= \prod_{i\geq 1} \frac{1+q^{2i}}{(1-q^{2i})(1+q^{2i-1})} \\
&=  \prod_{i\geq 1} \left( \frac{1-q^{4i}}{(1-q^{2i})^2}\right) \left(\frac{1+q^{2i}}{1+q^{i}}\right) \\
&=  \prod_{i\geq 1} \frac{(1-q^{4i})^2(1-q^i)}{(1-q^{2i})^4} \\
&\equiv   \prod_{i\geq 1} 1-q^i  \bmod{4} 
\end{align*}
using elementary generating function manipulations and, in the last step, applying \eqref{bin} to the denominator.  By \eqref{penteq}, we have
\begin{equation}
 \sum_{n=0}^\infty (-1)^n d_2(n)q^n \equiv \sum_{j=-\infty}^\infty (-1)^jq^{j(3j+1)/2}  \bmod{4}.  \label{neggf}
\end{equation}
We see immediately that $d_2(n) \equiv 0 \bmod{4}$ if $n$ is not a generalized pentagonal number.

There are four cases to consider when $n$ is a generalized pentagonal number.  
If $n = \omega_{4k}$, i.e., $n = (2k)(12k+1)$, 
then \eqref{neggf} gives
$$(-1)^{(2k)(12k+1)}d_2(\omega_{4k}) \equiv (-1)^{4k} \bmod{4}$$
and thus
$d_2(\omega_{4k}) \equiv 1 \bmod{4}$.

If $n = \omega_{4k+1}$, i.e., $n = (4k+1)(6k+2)$, 
then \eqref{neggf} gives
$$(-1)^{(4k+1)(6k+2)}d_2(\omega_{4k+1}) \equiv (-1)^{4k+1} \bmod{4}$$
and thus
$d_2(\omega_{4k+1}) \equiv -1 \bmod{4}$.

If $n = \omega_{4k+2}$, i.e., $n = (2k+1)(12k+7)$, 
then \eqref{neggf} gives
$$(-1)^{(2k+1)(12k+7)}d_2(\omega_{4k+2}) \equiv (-1)^{4k+2} \bmod{4}$$
and thus
$d_2(\omega_{4k+2}) \equiv -1 \bmod{4}$.

If $n = \omega_{4k+3}$, i.e., $n = (4k+3)(6k+5)$, 
then \eqref{neggf} gives
$$(-1)^{(4k+3)(6k+5)}d_2(\omega_{4k+3}) \equiv (-1)^{4k+3} \bmod{4}$$
and thus
$d_2(\omega_{4k+3}) \equiv 1 \bmod{4}$.
\end{proof}

For the combinatorial proof of Theorem \ref{d2mod4}, we do not partition $D_2(n)$ into disjoint 4-tuples when $n$ is not a generalized pentagonal number.  Instead, we consider $d^e_2(n)$ and $d^o_2(n)$, the number of partitions in $D_2(n)$ with even (respectively, odd) length, and carefully combine two near-bijections.  The data shown in Table \ref{d2eo} suggest our approach: The counts are almost all even and $d^e_2(n) = d^o_2(n)$ almost always.  The precise statements are given in the next two theorems.

\begin{table}[th]
\begin{tabular}{r|r|r}
$n$  & $d_2^e(n)$ & $d_2^o(n)$  \\ \hline
1 & 0 & 1  \\ 
2 & 1 & 2  \\ 
3 & 2 & 2  \\ 
4 & 4 & 4  \\ 
5 & 6 & 5  \\ 
6 & 10 & 10  \\ 
7 & 14 & 13  \\ 
8 & 22 & 22  \\ 
9 & 30 & 30  \\ 
10 & 46 & 46  \\ 
11 & 62 & 62  \\ 
12 & 91 & 92  \\ 
13 & 122 & 122  \\ 
14 & 174 & 174  \\ 
15 & 230 & 231  \\ 
\end{tabular}
\hskip .5in
\begin{tabular}{r|r|r}
$n$  & $d_2^e(n)$ & $d_2^o(n)$  \\ \hline
16 & 320 & 320  \\ 
17 & 420 & 420  \\ 
18 & 572 & 572  \\ 
19 & 744 & 744  \\ 
20 & 996 & 996  \\ 
21 & 1286 & 1286  \\ 
22 & 1697 & 1696  \\ 
23 & 2174 & 2174  \\ 
24 & 2834 & 2834  \\ 
25 & 3606 & 3606  \\ 
26 & 4651 & 4650  \\ 
27 & 5880 & 5880  \\ 
28 & 7512 & 7512  \\ 
29 & 9440 & 9440  \\ 
30 & 11962 & 11962  \\ 
\end{tabular}
\caption{The number of partitions of $D_2(n)$ by length parity through $n = 30$.} \label{d2eo}
\end{table}

The first of our two combinatorial results in this section incorporates the involutions of Glaisher and Franklin discussed in \S2.

\begin{theorem}
\label{mostlyeven} 
 For all $n \ge 0$,
 \begin{align*}
 d^e_2(n) & \equiv \begin{cases} 1 \bmod{2} & \text{if $n = \omega_{4k}$ or $n = \omega_{4k+1}$,} \\ 0 \bmod{2} & \text{otherwise;} \end{cases} \\
 d^o_2(n) & \equiv \begin{cases} 1 \bmod{2} & \text{if $n = \omega_{4k+2}$ or $n = \omega_{4k+3}$,} \\ 0 \bmod{2} & \text{otherwise.} \end{cases}
 \end{align*}
\end{theorem}

\begin{proof}
First, consider $D^e_2(n)$ when $n$ is not a generalized pentagonal number.  We describe an involution on $D^e_2(n)$ with no fixed points which allows us to conclude that $d^e_2(n)$ is even.  If a partition in $D^e_2(n)$ has any even parts, then toggle whether the largest even part is overlined.  This leaves partitions with only odd parts (so none are overlined), i.e., $\lambda \in P^e_o(n)$ (which occurs only when $n$ is even). We find its match by applying a composition of involutions: Glaisher's, Franklin's, and then Glaisher's again.  The first application of Glaisher's involution gives a partition in $Q(n)$.  Franklin's involution takes $G(\lambda)$ to a different partition in $Q(n)$.  The application of Glaisher's involution to $F(G(\lambda))$ gives a partition in $P^e_o(n)$ necessarily not equal to $\lambda$.  Because all elements of $D^e_2(n)$ are partitioned into disjoint pairs, $d^e_2(n)$ is even.

If $n$ is an odd generalized pentagonal number, i.e., when $n = \omega_{4k+2}$ or $n = \omega_{4k+3}$, then every partition in $D^e_2(n)$ has at least one even part and the toggling described above gives an involution with no fixed points on $D^e_2(n)$ so that, again, $d^e_2(n)$ is even.

This leaves the case that $n$ is an even generalized pentagonal number, i.e., $n = \omega_{4k}$ or $n = \omega_{4k+1}$.  Here, $P^e_o(n) = P_o(n)$ is nonempty.
For $n = \omega_{4k}$, the images of $P_o(n)$ under Glaisher's map include $(8k, 8k-1, \ldots, 4k+1)$ if $k$ is positive or $(8k-1, 8k-2, \ldots, 4k)$ if $k$ is negative, each a partition where Franklin's involution is not defined.  For $n = \omega_{4k+1}$, the problematic images of $P_o(n)$ under Glaisher's map are $(8k+2, 8k+1, \ldots, 4k+2)$ if $k$ is positive and $(8k+1, 8k, \ldots, 4k+1)$ if $k$ is negative.  In each of these cases, the exception to Franklin's involution means that one partition is unpaired in the involution and $d^e_2(n)$ is odd.

For $D^o_2(n)$, the same involution shows that $d^o_2(n)$ is even except when $n = \omega_{4k+2}$ or $n = \omega_{4k+3}$.  In those two cases, $n$ is an odd generalized pentagonal number and the image of $P^o_o(n) = P_o(n)$ under Glaisher's involution includes $(8k+4, 8k+3, \ldots, 4k+2)$ or $(8k+5, 8k+4, \ldots, 4k+3)$ (for $n = \omega_{4k+2}$) or $(8k+6, 8k+5, \ldots, 4k+3)$ or $(8k+7, 8k+6, \ldots, 4k+4)$ (for $n = \omega_{4k+3}$) for which Franklin's involution is not defined, so  $d^o_2(n)$ is odd in those cases.
\end{proof}

See Table \ref{mostlyevenex} for an example with both types of pairings (toggling whether the greatest even part is overlined and the composition of classical involutions) and one unmatched partition.

\begin{table}[th]
\renewcommand{\arraystretch}{1.25}
\begin{tabular}{c}
$(7) \g (7) \f (6,1) \g (3,3,1)$ \\
$(5,1,1) \g (5,2) \f (4,2,1) \g (1^7)$ \\
$(4,2,1) \sim (\o{4},2,1)$ \\
$(4,\o{2},1) \sim (\o{4}, \o{2}, 1)$ \\
$(3,2,2) \sim (3, \o{2}, 2)$ \\
$(3,1^4) \g (4,3) \f$  \textreferencemark \\
$(2,2,1^3) \sim (\o{2}, 2,1^3)$
\end{tabular}
\caption{Theorem \ref{mostlyeven}'s near-pairing of partitions in $D^o_2(7)$ where the unlettered $\sim$ indicates toggling the overline of the largest even part and \textreferencemark {} indicates a case where Franklin's involution is not defined.} \label{mostlyevenex}
\end{table}

The second of our two combinatorial proofs in this section uses the relatively new involution of van Leeuwen presented in \S 2 and, again, Franklin's.

\begin{theorem}
\label{mostlyequal}    
For all $n \ge 0$,
\[d^e_2(n) - d^o_2(n) = \begin{cases} (-1)^m & \text{if $n = \omega_m$,} \\ 0 & \text{otherwise.} \end{cases}\]
\end{theorem}

\begin{proof}
Write $\lambda \in D_2(n)$ as the bipartition $(\mu, \nu)$ where $\mu$ consists of any nonoverlined parts and $\nu$ corresponds to any overlined parts.  If $\mu$ is nonempty and an element of $R(m) = P(m) \setminus Q_o(m)$ for some $m \le n$, then match $\lambda$ with the element of $D_2(n)$ having bipartition $(L(\mu),\nu)$.  That is, apply van Leeuwen's map to the nonoverlined parts and fix any overlined parts.  Recall that the lengths of $\mu$ and $L(\mu)$ differ by one, thus the lengths of $\lambda$ and its image differ by one.

If $(\mu, \nu)$ associated with $\lambda$ has $\mu$ empty or $\mu$ consisting of distinct odd parts, then $\mu \cup \nu$ (i.e., $\lambda$ ignoring any overlines) is an element of $Q(n)$.  In these cases, apply Franklin's map to  $\mu \cup \nu$ and overline any even parts.  Recall that Franklin's map changes partition length by one.

This last operation fails exactly when $n$ is a generalized pentagonal number.  The detalis of these exceptions follow from \eqref{legeq}.
\end{proof}

We provide an example of the involution for $n = 8$.  Table \ref{ltab} above shows the pairings among the partitions with no overlined parts (i.e., $\nu = \varnothing$).  The remaining cases are given in Table \ref{mostlyequalex}, first the partitions with at least one overlined part and at least one nonoverlined part although not odd and distinct, then the partitions with only overlined parts or nonoverlined parts odd and distinct.

\begin{table}[th]
\renewcommand{\arraystretch}{1.25}
\begin{tabular}{r|l}
$D^e_2(n)$ & $D^o_2(n)$ \\ \hline
$(\o6,2)$ & $(\o6,1,1)$ \\
$(6,\o2)$ & $(3,3,\o2)$ \\ 
$(\o4,4)$ & $(\o4,2,2)$ \\
$(\o4,2,1,1)$ & $(\o4,1^4)$ \\ 
$(4,\o2,1,1)$ & $(4,\o2,2)$ \\
$(\o4,\o2,1,1)$ & $(\o4,\o2,2)$\\ 
$(3,\o2,2,1)$ & $(3,\o2,1,1,1)$ \\ 
$(\o2,2,2,2)$ & $(\o2,2,2,1,1)$ \\
$(\o2,2,1^4)$ & $(\o2,1^6)$ \\ \hline
$(7,1)$ & $(\o8)$ \\ 
$(\o6,\o2)$ & $(5,\o2,1)$ \\
$(5,3)$ & $(\o4,3,1)$
\end{tabular}
\caption{Theorem \ref{mostlyequal}'s involution on $D_2(8)$ includes the pairings given in Table \ref{ltab} and the ones listed here.} \label{mostlyequalex}
\end{table}

We now combine the results of Theorem \ref{mostlyeven} and Theorem \ref{mostlyequal} to give a combinatorial proof of Theorem \ref{d2mod4}.

\begin{proof}[Second proof of Theorem \ref{d2mod4}]  For $n$ not a generalized pentagonal number, we know from Theorem \ref{mostlyequal} that $d^e_2(n) = d^o_2(n)$ and from Theorem \ref{mostlyeven} that $d^e_2(n) = 2j$ for some integer $j$.  Therefore
\[d_2(n) = d^e_2(n) + d^o_2(n) = 2j + 2j = 4j\]
as desired.

As in the generating function proof of Theorem \ref{d2mod4}, there are four cases to consider when $n$ is a generalized pentagonal number.
If $n = \omega_{4k}$, then $d^e_2(\omega_{4k}) - d^o_2(\omega_{4k}) = 1$ by Theorem \ref{mostlyequal} and $d^o_2(\omega_{4k}) = 2j$ for some integer $j$ by Theorem \ref{mostlyeven}, therefore
\[d_2(\omega_{4k}) = d^e_2(\omega_{4k}) + d^o_2(\omega_{4k}) = (2j+1) + 2j = 4j + 1 \equiv 1 \bmod{4}.\]

If $n = \omega_{4k+1}$, then $d^e_2(\omega_{4k+1}) - d^o_2(\omega_{4k+1}) = -1$ by Theorem \ref{mostlyequal} and $d^o_2(\omega_{4k+1}) = 2j$ for some integer $j$ by Theorem \ref{mostlyeven}, therefore
\[d_2(\omega_{4k+1}) = d^e_2(\omega_{4k+1}) + d^o_2(\omega_{4k+1}) = (2j-1) + 2j = 4j - 1 \equiv 3 \bmod{4}.\]

If $n = \omega_{4k+2}$, then $d^e_2(\omega_{4k+2}) - d^o_2(\omega_{4k+2}) = 1$ by Theorem \ref{mostlyequal} and $d^e_2(\omega_{4k+2}) = 2j$ for some integer $j$ by Theorem \ref{mostlyeven}, therefore
\[d_2(\omega_{4k+2}) = d^e_2(\omega_{4k+2}) + d^o_2(\omega_{4k+2}) = 2j + (2j-1) = 4j - 1 \equiv 3 \bmod{4}.\]

If $n = \omega_{4k+3}$, then $d^e_2(\omega_{4k+3}) - d^o_2(\omega_{4k+3}) = -1$ by Theorem \ref{mostlyequal} and $d^e_2(\omega_{4k+3}) = 2j$ for some integer $j$ by Theorem \ref{mostlyeven}, therefore
\[d_2(\omega_{4k+3}) = d^e_2(\omega_{4k+3}) + d^o_2(\omega_{4k+3}) = 2j + (2j+1) = 4j + 1 \equiv 1 \bmod{4}.\qedhere \]
\end{proof}

\section*{Acknowledgements}
This work was initiated while the first author was a participant in the University of Minnesota Duluth's Undergraduate Research Opportunities Program (UROP) during the Fall 2023 semester.  We gratefully acknowledge the support of the University of Minnesota Duluth.

\end{document}